\documentclass[12pt]{amsart}
\usepackage{amssymb}
\usepackage{amsfonts}
\usepackage{amsmath}
\usepackage{graphicx}
\usepackage{xcolor}
\usepackage{mathrsfs}
\usepackage{stmaryrd}
\usepackage{epsfig,color}
\usepackage{blindtext}
\usepackage{enumerate}
\usepackage{hyperref}
\usepackage{url}
\usepackage{bbm}
\usepackage{filecontents}
\usepackage{nicefrac,mathtools}
\usepackage{bm}  
\usepackage[nocompress]{cite}
\DeclareGraphicsExtensions{.pdf,.jpeg,.png}
\usepackage{epstopdf}
\usepackage{cancel} 
\usepackage[normalem]{ulem} 
\usepackage{verbatim} 
\usepackage{enumitem} 
\usepackage{tikz-cd}
\usetikzlibrary{cd}
\usepackage{color}
\usepackage[msc-links, lite]{amsrefs}
\usepackage{geometry}
\usepackage{tikz}
\usetikzlibrary{decorations.markings}
\usetikzlibrary{arrows.meta}

\usepackage{extarrows}
\newtheorem{theorem}{Theorem}[section]

\newtheorem{proposition}[theorem]{Proposition}
\newtheorem{lemma}[theorem]{Lemma}
\newtheorem{corollary}[theorem]{Corollary}
\newtheorem{question}[theorem]{Question}
\newtheorem{example}[theorem]{Example}

\theoremstyle{definition}
\newtheorem{definition}[theorem]{Definition}

\numberwithin{equation}{section}

\numberwithin{equation}{section}

\makeatletter
\let\save@mathaccent\mathaccent
\newcommand*\if@single[3]{%
\setbox0\hbox{${\mathaccent"0362{#1}}^H$}%
\setbox2\hbox{${\mathaccent"0362{\kern0pt#1}}^H$}%
\ifdim\ht0=\ht2 #3\else #2\fi
}

\makeatother

\makeatletter
\newcommand*{\transpose}{%
{\mathpalette\@transpose{}}%
}
\newcommand*{\@transpose}[2]{%
\raisebox{\depth}{$\m@th#1\intercal$}%
}
\makeatother
\usetikzlibrary{hobby}

 \usetikzlibrary{decorations}
\makeatletter
\def\pgfutil@Repeat#1#2{#2\ifnum#1>0
  \expandafter\pgfutil@firstofone\else\expandafter\pgfutil@gobble\fi
  {\expandafter\pgfutil@Repeat\expandafter{\the\numexpr#1-1\relax}{#2}}}
\tikzset{
  dash between/.code args={#1 and #2}{%
    \tikz@addoption{%
      \pgfgetpath\currentpath
      \pgfprocessround{\currentpath}{\currentpath}%
      \pgf@decorate@parsesoftpath{\currentpath}{\currentpath}%
      \pgfmathsetlengthmacro\firstpart{(#1)*\pgf@decorate@totalpathlength}%
      \pgfmathsetlengthmacro\secondpart{(#2-(#1))*\pgf@decorate@totalpathlength}%
      \pgfmathsetlengthmacro\thirdpart{(1-(#2))*\pgf@decorate@totalpathlength}%
      \edef\thirdpart{{\thirdpart}{0pt}}%
      \edef\firstpart{{\firstpart}{0pt}}%
      \pgfmathsetlengthmacro\secondpartlength{\pgfkeysvalueof{/tikz/dash between on}
                                            +(\pgfkeysvalueof{/tikz/dash between off})}%
      \pgfmathtruncatemacro\repetitions{\secondpart/\secondpartlength}%
      \pgfmathsetlengthmacro\secondexpand{\secondpart/\repetitions-\secondpartlength}%
      \edef\secondexpand{\the\dimexpr\pgfkeysvalueof{/tikz/dash between off}+\secondexpand\relax}%
      \edef\secondpart{%
        \pgfutil@Repeat{\the\numexpr\repetitions-1\relax}%
          {{\pgfkeysvalueof{/tikz/dash between on}}{\secondexpand}}%
      }%
      \edef\tikz@temp{\firstpart\secondpart\thirdpart}%
      \expandafter\pgfsetdash\expandafter{\tikz@temp}{+0pt}%
    }
  }
}
\makeatother

\tikzset{
  dash between style/.is choice,
  dash between style/dotted/.style        ={dash between on=\pgflinewidth,dash between off=2pt},
  dash between style/densely dotted/.style={dash between on=\pgflinewidth,dash between off=1pt},
  dash between style/loosely dotted/.style={dash between on=\pgflinewidth,dash between off=4pt},
  dash between style/dashed/.style        ={dash between on=3pt,dash between off=2pt},
  dash between style/loosely dashed/.style={dash between on=3pt,dash between off=6pt},
  dash between style/densely dashed/.style={dash between on=3pt,dash between off=2pt},
  dash between style/no/.style={dash between on=0pt, dash between off=1pt},
  dash between on/.initial=\pgflinewidth,
  dash between off/.initial=2pt,
  middle dotted line/.style={
    thick,
    dash between=.35 and .65}}

\newcommand\QQ{\mathbb{Q}}
\newcommand\CC{\mathbb{C}}

\newcommand\RR{\mathbb{R}}

\newcommand\ZZ{\mathbb{Z}}

\DeclareMathOperator{\et}{\acute{e}t}

\DeclareMathOperator{\pt}{pt}

\DeclareMathOperator{\Gal}{Gal}

\DeclareMathOperator{\Spec}{Spec}

\newcommand{\legendre}[2]{\ensuremath{\left( \frac{#1}{#2} \right) }}

\title{Galois symmetry of $\Gal(\overline{\QQ}/\QQ)$ on topological manifold structures of varieties}
\author{Runjie Hu}
\newcommand{\Addresses}{{
  \bigskip
  \footnotesize

  Runjie Hu, \textsc{Department of Mathematics, Texas A\&M University, College Sta, TX 77843}\par\nopagebreak
  \textit{E-mail address}, \texttt{ runjie.hu@tamu.edu}

}}
\date{}
\begin{document}

\maketitle

\begin{abstract}
We propose a definition of the profinite normal structure set for the set of all manifolds in a fixed profinite homotopy type. Using this framework, we prove that the Galois action of $\Gal(\overline{\QQ}/\QQ)$ on the underlying topological manifold structures of smooth, complete, simply-connected complex varieties defined over $\overline{\QQ}$ of dimension at least $3$ factors through the abelianization of $\Gal(\overline{\QQ}/\QQ)$. Moreover, this abelian action extends canonically to the entire profinite normal structure set. This result provides an answer to the question by Sullivan in the case of topological manifold structures of simply-connected varieties.
\end{abstract}

\section{Introduction}

For any variety defined over the algebraic closure $\overline{\QQ}$ of $\QQ$, a Galois conjugation of $\overline{\QQ}$ not only produces a Galois conjugate variety by conjugating the defining polynomials, but also induces a canonical map between the original variety and its Galois conjugate by conjugating $\overline{\QQ}$-points. One may extend this Galois automorphism of $\overline{\QQ}$ to $\CC$. The extending map on the complex points of the varieties is usually discontinuous. 

However, this ``discontinous'' map becomes a homotopy equivalence after profinite completion of the varieties (\cite{Artin&Mazur1969}). Recall that the profinite completion of a space captures the finite information of the underlying homotopy type. Consequently, the finite information (known as the profinite completion) of the fundamental group is Galois invariant. However, topological invariants that rely solely on the transcendental topology of complex numbers are generally not Galois invariant. For instance, the entire fundamental group of a variety is not Galois invariant (\cite{Serre1964}). Thus the underlying profinite homotopy type of a variety is algebraic rather than transcendental.

The Galois conjugating map between the variety and its Galois conjugate aligns with the concept of the structure set of a manifold. Recall that the structure set of a manifold consists of homotopy classes of all homotopy equivalences from some other manifold into this manifold, modulo homeomorphisms on the domains. Intuitively the structure set represents the set of all topological manifolds in a homotopy type. It is necessary to formalize the finite information of a structure set to incorporate Galois conjugating maps of varieties and study the Galois action of $\Gal(\overline{\QQ}/\QQ)$ on the underlying manifold structures of varieties.  

The naive approach, defining the profinite structure set as the set of homotopy classes of all profinite homotopy equivalences from some other manifold to a fixed manifold, modulo homeomorphisms on the domains, does not inherently carry a Galois symmetry. To overcome this technical difficulty, we introduce the \textbf{profinite normal structure set} (see Definition \ref{definition-of-profinite-structure-set}), defined as the set of different profinite liftings of the Spivak normal spherical fibration to topological $\RR^n$-bundles.

Within the profinite normal structure set, there is a subset consisting of topological manifolds that are represented by complex varieties defined over $\overline{\QQ}$. These elements are referred to as \textbf{$\overline{\QQ}$-algebraic elements} (see the paragraph above Definition \ref{definition-of-Q-bar-algebraic-elements}). The Galois conjugating maps of varieties induce a Galois action of $\Gal(\overline{\QQ}/\QQ)$ on the $\overline{\QQ}$-algebraic elements (see Definition \ref{definition-of-Q-bar-algebraic-elements}). This action corresponds to the Galois action of $\Gal(\overline{\QQ}/\QQ)$ on the underlying manifold structures of varieties.

A vague statement of our theorem is the following (for a more rigorous statement, see Theorem \ref{GaloisMain}).

\begin{theorem}\label{Main-Introduction}
For a simply-connected, compact, smooth complex variety $X$ defined over $\overline{\QQ}$ of dimension at least $3$, the Galois action of $\Gal(\overline{\QQ}/\QQ)$ on the underlying topological manifold structures of the $\overline{\QQ}$-varieties which are profinite homotopy equivalent to $X$ (i.e., the $\overline{\QQ}$-algebraic elements) factors through the abelianization $\widehat{\ZZ}^{\times}$ of $\Gal(\overline{\QQ}/\QQ)$. Moreover, this abelian action of $\widehat{\ZZ}^{\times}$ canonically extends to the entire profinite normal structure set of $X$.
\end{theorem}

Our result answers the question by Sullivan in the case of underlying topological manifold structure of simply-connected varieties of higher dimensions.

\begin{question}[\cite{SullivanMITnotes}*{p.~271, Problem 4}]
Analyze the action of $\Gal(\overline{\QQ}/\QQ)$ on the manifold structures in a profinite homotopy type associated to nonsingular
algebraic varieties defined over $\overline{\QQ}$.
\end{question}

The proof of Theorem \ref{Main-Introduction} relies on a successful application of algebraic geometry to the Adams conjecture in topology (\cite{Quillen1968}\cite{Friedlander-thesis}\cite{Sullivan1974}). The Adams conjecture states that the underlying profinite homotopy sphere bundle of a vector bundle remains invariant under Adams operations modified in a certain way. A key point of the proof in \cite{Sullivan1974} lies in the construction of unstable Adams operations on finite Grassmannians. These unstable Adams operations arise naturally from the Galois action $\Gal(\overline{\QQ}/\QQ)$ on the Grassmannians, as Grassmannians are algebraic varieties defined over $\QQ$.

We first construct an abelianized Galois action of $\widehat{\ZZ}^{\times}$ on the profinite normal structure set of a variety (see Definition \ref{Definition-of-abelian-Galois-symmetry}). At odd primes, the action is induced by Adams operations on the tangent bundles of manifolds; at prime $2$, additional efforts are required (see Definition \ref{definition-of-Galois-k-classes} and the discussions above it). The proof of Theorem \ref{Main-Introduction} reduces to proving the compatibility between the Galois action of $\Gal(\overline{\QQ}/\QQ)$ on a variety and the abelianized Galois action of $\widehat{\ZZ}^{\times}$ on the profinite normal structure set of a variety, via the abelianization map $\Gal(\overline{\QQ}/\QQ)\rightarrow \widehat{\ZZ}^{\times}$. 

The most challenging part in proving the compatibility is to handle the prime $2$ information of the profinite normal structure set. To overcome this technical difficulty, we use the sophisticated results in \cite{Brumfiel&Morgan}\cite{Sullivan&Morgan}\cite{Madsen&Milgram1974} about the $2$-local characteristic classes of homotopy sphere bundles and $TOP$ bundles.

Our result shows the algebraic aspect of Galois symmetry on manifold structures, but the geometric aspect still remains mysterious (see comments and discussions in \cite{SullivanMITnotes}*{p.~271-272} and \cite{Grothendieck-dessins-denfants}). The discussions about dessins d'enfants (see \cite{Survey-Absolute-Galois-Action}\cite{Lochak-Schneps}\cite{Ihara-Matsumoto}) suggest a potential pathway to understanding the geometric aspect of Galois symmetry. In future works, we will provide a geometric and combinatorial explanation to explain the Galois symmetry on the manifold structures of varieties.

\subsubsection*{Acknowledgements.} This research is conducted during my PhD studies at Stony Brook University and my postdoctoral studies at Texas A\&M University. It is partially supported by the Simons Foundation International and NSF Grant 1952693. I would like to express my gratitude to my thesis advisor, Dennis Sullivan, for introducing me to this problem. I also want to thank Mark de Cataldo, James F. Davis, Jiahao Hu, John Morgan, James Myer, John Pardon, Jason Starr, Guozhen Wang, Shmuel Weinberger, Zhouli Xu and Siqing Zhang for many useful discussions and valuable suggestions.

\section{profinite normal structure set and abelianized Galois symmetry}

Let $X$ be a simply-connected closed topological manifold of dimension at least $5$. Recall that the structure set $S(X)$ (e.g., \cite{Weinberger1995StratifiedSpaces}*{p.~57}\cite{davis}*{Definition 4.1}) is the set of homotopy classes of all homotopy equivalences $M\rightarrow X$ with $M$ a closed topological manifold modulo the following equivalence relation. $M\rightarrow X$ and $M'\rightarrow X$ are equivalent if there exists a homeomorphism $M\rightarrow M'$ such that the following diagram commutes up to homotopy.
\[
\begin{tikzcd}
    M \arrow[r] \arrow[d] & X \\
    M' \arrow[ru] &
\end{tikzcd}
\]

A naive idea to profinite complete the structure set $S(X)$ is the set $S(X)^{\wedge}_G$ of homotopy classes of all homotopy equivalences of profinite completions $M^{\wedge}\rightarrow X^{\wedge}$ with $M$ a closed topological manifold modulo the following equivalence relation. $M^{\wedge}\rightarrow X^{\wedge}$ and $(M')^{\wedge}\rightarrow X^{\wedge}$ are equivalent if there exists a homeomorphism $M\rightarrow M'$ such that the following diagram commutes up to homotopy.
\[
\begin{tikzcd}
    M^{\wedge} \arrow[r] \arrow[d] & X^{\wedge} \\
    (M')^{\wedge} \arrow[ru] &
\end{tikzcd}
\]

However, it is hard to construct the abelianized Galois action $\widehat{\ZZ}^{\times}$ on $S(X)^{\wedge}_G$. So we use the following indirect way to define the profinite structure set.

Note that any homotopy equivalence $f:M^{\wedge}\rightarrow X^{\wedge}$ provides a fiberwise homotopy equivalence of profinite spherical fibrations $(f^{-1})^*\nu^{\wedge}_M\rightarrow \nu_X$ over $X$, where $\nu_M$ and $\nu_X$ are stable normal bundles of $M$ and $X$ respectively. This corresponds to an element in $[X,(G/TOP)^{\wedge}]$, where $(G/TOP)^{\wedge}$ is the profinite completion of the $TOP$ surgery space $G/TOP$.

\begin{definition}\label{definition-of-profinite-structure-set}
Let $X$ be a simply-connected closed topological manifold of dimension at least $5$. The \textbf{profinite normal structure set} $S(X)^{\wedge}_N$ is defined to be $[X,(G/TOP)^{\wedge}]$.
\end{definition}

$(G/TOP)^{\wedge}$ is homotopy equivalent to the product of $p$-completions $(G/TOP)^{\wedge}_p$ of $G/TOP$ for all primes $p$. Recall that $(G/TOP)^{\wedge}_{2}\simeq \prod_{k>0}(K(\ZZ^{\wedge}_{2},4k)\times K(\ZZ/2,4k-2))$ and $(G/TOP)^{\wedge}_{p}\simeq BSO^{\wedge}_{p}$ when $p$ is odd (\cite{Sullivan1996}*{p.~85, Theorem 4},\cite{KirbySiebenmann}*{p.~329, 15.3}). This directly deduce the following.

\begin{proposition}
$S(X)^{\wedge}_N$ is bijective to the set of $(\{\phi_p\}_{\text{(odd $p$)}},l,k)$, where $\{\phi_p\}_{\text{(odd $p$)}}\in \prod_{\text{(odd $p$)}} \widetilde{KO_p^\wedge}(X)$, $l\in H^{4*}(X;\ZZ^{\wedge}_{2})$ and $k\in H^{4*-2}(X;\ZZ/2)$.   
\end{proposition}

Recall from \cite{Brumfiel&Morgan}*{Theorem E and p.~9}\cite{SullivanMITnotes}*{Theorem 6.5}\cite{Sullivan&Morgan}*{p.~530} that the stable normal bundle $\nu_X$ of $X$ has a canonical Thom class $(\Delta_X)_p\in\widetilde{KO_p^\wedge}(\widehat{M(\nu_X)}_p)$ for each odd $p$ and a characteristic class $L_X\in H^{4*}(X;\widehat{\ZZ}_2)$.

\begin{definition}\label{Definition-of-abelian-Galois-symmetry}
The \textbf{abelianized Galois symmetry $\widehat{\ZZ}^\times$ on $\mathbf{S}(X)^\wedge_N$} is defined by 
\begin{equation}\label{odd-prime-Galois-symmetry}
\phi_p \rightarrow \psi^{\sigma_p}\phi_p\cdot \frac{\psi^{\sigma_p}(\Delta_X)_p}{(\Delta_X)_p}  \end{equation}   
\begin{equation}\label{2-Galois-symmetry-l-class}
(1+8\cdot\sigma_2(l)) \cdot L_X=(1+8\cdot\psi^{\sigma_2}_Hl)\cdot \psi^{\sigma_2}_HL_X  
\end{equation}
\begin{equation}\label{2-Galois-symmetry-k-class}
\sigma_2(k)=k+k^{\sigma_2}_X
\end{equation}
where $\phi_p\in \widetilde{KO_p^\wedge}$, $l\in H^{4*}(X;\ZZ^{\wedge}_{2})$, $k\in H^{4*-2}(X;\ZZ/2)$, $(\sigma_p)_p\in \prod_p \widehat{\ZZ}_p^\times\cong \widehat{\ZZ}^\times$, $\psi^{\sigma_p}$ is the Adams operation, $\psi^{\sigma_2}_H$ is the cohomological Adams operation on $H^{2*}(X;\widehat{\ZZ}_2)$ and $k^{\sigma_2}_X$ is constructed below.   
\end{definition}

Recall from \cite{Sullivan1974}\cite{Madsen&Milgram}*{p.~106} that the map $BU^{\wedge}_2\xrightarrow{\psi^{\sigma_2}-1} BU^{\wedge}_2$ canonically factors as $BU^{\wedge}_2\rightarrow (G/U)^{\wedge}_2\rightarrow BU^{\wedge}_2$. Then there is a composition of maps $f_{\sigma_2}:BU^{\wedge}_2\rightarrow (G/U)^{\wedge}_2\rightarrow (G/TOP)^{\wedge}_2$. Let $k^q\in H^{4*+2}((G/TOP)^{\wedge}_2;\ZZ/2)$ be the Kervaire characteristic class (\cite{Sullivan&Rourke}\cite{Sullivan1996}*{p.~88, Corollary 1}).
Let $\widetilde{k}^{\sigma_2}=f^*_{\sigma_2}k^q$.

Since $k^{\sigma_2}$ is a combination of Stiefel-Whitney classes, it can be lifted to the same combination of Stiefel-Whitney classes of spherical fibrations, i.e., $\widetilde{k}^{\sigma_2}$ actually lies in $H^{4*+2}(BG_2;\ZZ/2)$, where $BG_2$ is the classifying space of $2$-profinite spherical fibrations (\cite{SullivanMITnotes}*{Theorem 4.2}). Let $k^{\sigma_2}$ be this characteristic class in $H^{4*+2}(BG_2;\ZZ/2)$.

\begin{definition}\label{definition-of-Galois-k-classes}
Define $k^{\sigma_2}_X=(\nu_X)^* k^{\sigma_2}\in H^{4*+2}(X;\ZZ/2)$, where $\nu_X:X\rightarrow BG$ is the underlying spherical fibration of the stable normal bundle of $X$.    
\end{definition}

In the proof of our main theorems \ref{GaloisMain}\ref{GaloisMainIII}, we need the following lemma. Its proof is postposted to Section \ref{algorithm-for-k-class}. In Section \ref{algorithm-for-k-class} we also compute the characteristic classes $k^{\sigma_2}$.

\begin{lemma}(Additivity of $k^{\sigma_2}$)\label{Additivity-of-k-classes}\\
Let $\Delta: BU^{\wedge}_2\times BU^{\wedge}_2\rightarrow BU^{\wedge}_2$ be the $H$-space product induced by the Whitney sum of vector bundles. Then $\Delta^* k^{\sigma_2}=k^{\sigma_2}\times 1+ 1 \times k^{\sigma_2}$.
\end{lemma}

\section{Main Theorems}

Let $X,Y$ be smooth complete complex varieties. Let $X_{\et},Y_{\et}$ be their \'etale homotopy types (technically, we take the homotopy inverse limit of the pro-spaces representing their \'etale homotopy types). By \cite[Corollary 12.10]{Artin&Mazur1969}, $X_{\et}$ is homotopy equivalent to the profinite completions $X^{\wedge}$ of $X$. The same is true for $Y_{\et}$.

Assume that $Y$ is simply-connected. Let $f:X\rightarrow Y$ be an algebraic map over some field automorphism of $\CC$ such that $f_{\et}:X_{\et}\rightarrow Y_{\et}$ is a homotopy equivalence. 

Then $f_{\et}$ represents an element in the profinite normal structure set $S(Y)^{\wedge}_N$. Call such an element \textbf{a complex algebraic element in $\mathbf{S}(Y)^\wedge_N$}.

Recall that any $\sigma\in \Gal(\CC/\QQ)$ induces an algebraic map $\sigma:X^{\sigma}\rightarrow X$, which conjugates the $\CC$-points. By \cite[Corollary 12.11]{Artin&Mazur1969}, $\sigma:X^{\sigma}\rightarrow X$ induces a homotopy equivalence $\sigma:X^{\sigma}_{\et}\rightarrow X_{\et}$.

\begin{definition}
The \textbf{Galois action of $\Gal(\CC/\QQ)$ on the complex algebraic elements in $\mathbf{S}^{TOP}(Y)^\wedge$ (or $\mathbf{S}^{TOP}(Y)^\wedge_p$)} is defined by
\[
\sigma:(X\xrightarrow{f} Y) \rightarrow ((X^{\sigma})\xrightarrow{\sigma} X\xrightarrow{f} Y)
\]
where $\sigma\in \Gal(\CC/\QQ)$.   
\end{definition}

Let $\omega:\Gal(\CC/\QQ)\rightarrow \widehat{\ZZ}^{\times}$ be the restriction of the field automorphisms of $\CC$ on the roots of unity.

\begin{theorem}\label{GaloisMainIII}
Let $Y$ be a smooth, complete, connected, simply-connected complex variety. Assume that $Y$ has complex dimension at least $3$. Then the Galois action $\Gal(\CC/\QQ)$ on the complex algebraic elements in $\mathbf{S}(Y)^\wedge_N$ factors through the homomorphism $\omega:\Gal(\CC/\QQ)\rightarrow \widehat{\ZZ}^{\times}$ given by restriction on roots of unity. Moreover, this abelian action canonically extends to the entire $\mathbf{S}(Y)^\wedge_N$.
\end{theorem}

We prove this theorem by showing the following lemma.

\begin{lemma}\label{GaloisMainII}
With the same assumption as in Theorem \ref{GaloisMainIII}, the Galois action $\Gal(\CC/\QQ)$ on the complex algebraic elements of $\mathbf{S}(Y)^\wedge_N$ agrees with the abelianized Galois action of $\widehat{\ZZ}^\times$ on $\mathbf{S}(Y)^\wedge_N$ via $\omega:\Gal(\CC/\QQ)\rightarrow \widehat{\ZZ}^{\times}$.
\end{lemma}

We prove this lemma in Section \ref{Proof-of-Main-Theorem}.

Compared to the representation of $\Gal(\CC/\QQ)$, we are more interested in the Galois actions of $\Gal(\overline{\QQ}/\QQ)$. For this, fix a field embedding $\overline{\QQ}\rightarrow \CC$.

Let $X,Y$ be smooth complete varieties defined over $\overline{\QQ}$ and $X_{\CC},Y_{\CC}$ be the corresponding complex varieties. By \cite{Artin&Mazur1969}*{Corollary 12.12}, there is a canonical homotopy equivalence $(X_{\CC})_{\et}\simeq X_{\et}$. The same holds for $Y$.

Assume that $Y_{\CC}$ is simply-connected. Let $f:X\rightarrow Y$ be an algebraic morphism over some field automorphism of $\overline{\QQ}$ such that $f_{\et}:X_{\et}\rightarrow Y_{\et}\simeq (Y_{\CC})_{\et}$ is a homotopy equivalence. Call such an element $f$ in  $\mathbf{S}(Y_{\CC})^\wedge_N$ an \textbf{$\overline{\QQ}$-algebraic element} in the profinite normal structure set $\mathbf{S}(Y_{\CC})^\wedge_N$.

\begin{definition}\label{definition-of-Q-bar-algebraic-elements}
The \textbf{Galois symmetry $\Gal(\overline{\QQ}/\QQ)$ on the $\overline{\QQ}$-algebraic elements in $\mathbf{S}(Y_{\CC})^\wedge_N$} is defined by
\[
\sigma:(X\xrightarrow{f} Y) \rightarrow ((X^{\sigma})\xrightarrow{\sigma} X\xrightarrow{f} Y)
\]
where $\sigma\in \Gal(\overline{\QQ}/\QQ)$. 
\end{definition}

Note that the restricting homomorphism $\omega:\Gal(\CC/\QQ)\rightarrow \widehat{\ZZ}^{\times}$ factors as $\Gal(\CC/\QQ)\rightarrow \Gal(\overline{\QQ}/\QQ)\xrightarrow{\omega'}\widehat{\ZZ}^{\times}$. The Kronecker-Weber theorem in the class field theory states that the homomorphism $\omega'$ is the abelianization of $\Gal(\overline{\QQ}/\QQ)$.

An algebraic map $f:X\rightarrow Y$ of $\overline{\QQ}$-varieties can be extended to an algebraic map $f_{\CC}:X_{\CC}\rightarrow Y_{\CC}$ of complex varieties. Indeed, since $f$ satisfies a commutative diagram for some $\tau\in \Gal(\overline{\QQ}/\QQ)$,
\[
\begin{tikzcd}
    X \arrow[r,"f"] \arrow[d] & Y \arrow[d] \\
    \Spec(\overline{\QQ}) \arrow[r,"\tau"] & \Spec(\overline{\QQ})
\end{tikzcd}
\]
one may extend $\tau$ to a field automorphism $\widetilde{\tau}\in \Gal(\CC/\QQ)$ and $f_{\CC}$ is defined over $\widetilde{\tau}$ by the universal property of pullback.

\begin{lemma}
The extended algebraic map $f_{\CC}$ in $\mathbf{S}(Y_{\CC})^\wedge_N$ is independent of the choices of the extensions $\widetilde{\tau}$ of $\tau$.
\end{lemma}

\begin{proof}
Consider the following commutative diagram.
\[
\begin{tikzcd}
    X_{\CC} \arrow[r,"f_{\CC}"] \arrow[d] & Y_{\CC} \arrow[d] \\
    X \arrow[r,"f"] & Y
\end{tikzcd}
\]
The algebraic element $(X_{\CC})_{\et}\xrightarrow{(f_{\CC})_{\et}}(Y_{\CC})_{\et}$ in $\mathbf{S}^{TOP}(Y_{\CC})^\wedge$ is the composition of maps $(X_{\CC})_{\et}\xrightarrow{\simeq} X_{\et} \xrightarrow{f} Y_{\et} \xleftarrow{\simeq}(Y_{\CC})_{\et}$. But the latter  map is independent of the choice of $\widetilde{\tau}$.
\end{proof}

Similarly, extend $\sigma\in \Gal(\overline{\QQ}/\QQ)$ to some $\widetilde{\sigma}\in\Gal(\CC/\QQ)$. The following commutative diagram proves that the algebraic element $\widetilde{\sigma}(f_{\CC})$ in $\mathbf{S}^{TOP}(Y_{\CC})^\wedge$ is independent of the choices of the liftings $\widetilde{\sigma}$ of $\sigma$.
\[
\begin{tikzcd}
    (X^{\sigma}_{\CC})_{\et} \arrow[r,"\widetilde{\sigma}"] \arrow[d,"\simeq"] & (X_{\CC})_{\et} \arrow[r,"f_{\CC}"] \arrow[d,"\simeq"] & (Y_{\CC})_{\et} \arrow[d,"\simeq"] \\
    (X^{\sigma})_{\et} \arrow[r,"\sigma"] & X_{\et} \arrow[r,"f_{\CC}"] & Y_{\et} 
\end{tikzcd}
\]

\begin{proposition}
There is an equivariant map from the subset of $\overline{\QQ}$-algebraic elements in $\mathbf{S}(Y_{\CC})^\wedge_N$ with $\Gal(\overline{\QQ}/\QQ)$-action to that of $\CC$-algebraic elements in $\mathbf{S}(Y_{\CC})^\wedge_N$ with $\Gal(\CC/\QQ)$-action, via the canonical quotient map $\Gal(\CC/\QQ)\rightarrow \Gal(\overline{\QQ}/\QQ)$.    
\end{proposition}

Then the following theorem is a corollary of  \ref{GaloisMainIII}.

\begin{theorem}\label{GaloisMain}
Let $Y$ be a smooth, complete, connected, simply-connected variety over $\overline{\QQ}$. Assume that $Y$ has dimension at least $3$. The Galois action $\Gal(\overline{\QQ}/\QQ)$ on the $\overline{\QQ}$-algebraic elements in $\mathbf{S}(Y_{\CC})^\wedge_N$ factors through the abelianization homomorphism $\omega':\Gal(\overline{\QQ}/\QQ)\rightarrow\widehat{\ZZ}^{\times}$. Moreover, this abelian action canonically extends to the entire $\mathbf{S}(Y_{\CC})^\wedge_N$.
\end{theorem}

This is Theorem \ref{Main-Introduction} in the Introduction.

\section{Proof of Lemma \ref{GaloisMainII}}\label{Proof-of-Main-Theorem}

The proof is based on the following lemma.

\begin{lemma}
Let $\sigma\in \Gal(\CC/\QQ)$. Let $X$ be a smooth complex variety and $TX$ the tangent bundle of $X$. Then the following diagram commutes up to homotopy, where $\omega:\Gal(\CC/\QQ)\rightarrow \widehat{\ZZ}^{\times}$ is the restriction to the roots of unity.
\[
\begin{tikzcd}
 \widehat{X} \arrow[r,"TX"] \arrow[d,"\sigma^{-1}"] & \widehat{BU} \arrow[d,"\psi^{\omega(\sigma)}"] \\
 \widehat{X^{\sigma}} \arrow[r,"TX^{\sigma}"] & \widehat{BU}
\end{tikzcd}
\]
\end{lemma}

This lemma can be easily proven by the \'etale realizations of the algebraic stack $[ */GL(n,\CC) ]$ (see \cite{Chough-Etale-Homotopy-Algebraic-Stacks}), the comparison theorem of the \'etale homotopy type with the profinite completion of the topological space $BGL(n,\CC)$ (\cite{Chough-Etale-Homotopy-Algebraic-Stacks}*{Theorem 4.3.24}) and the proof of the Adams conjecture in \cite{Sullivan1974}. We provide a more elementary proof with the idea of \cite{Deligne&Sullivan1975}*{Paragraph below the Proposition}.

\begin{proof}
Let $STX$ be the associated Stiefel bundle of $TX$ over $X$. That is, $STX$ is the space of all embeddings of vector spaces $T_xX\rightarrow \CC^N$, where $x$ ranges over all points of $X$ and $N$ is a sufficiently large number. The bundle $STX$ is a complex algebraic bundle over $X$. Each fiber is isomorphic to the Stiefel variety $V_{n,N}$ consisting of all $n$-frames in $\CC^N$.

There is a complex algebraic morphism $STX\rightarrow Gr_{n}(\CC^N)$, which takes an embedding $T_xX\rightarrow \CC^N$ to its image, where $n$ is the complex dimension of $X$.

Then we get a zig-zag of algebraic maps of varieties $X \leftarrow STX \rightarrow Gr_{n}(\CC^N)$. Notice that the left arrow is homotopically highly connected if $N$ is large enough. When $N$ tends to infinity, the zig-zag becomes the classifying map $X\rightarrow BGL(n,\CC)$ of the tangent bundle $TX_{\CC}$. This proves that $X_{\et}\simeq \widehat{X}\xrightarrow{TX} BGL(N,\CC)^{\wedge}\rightarrow BGL(\infty,\CC)^{\wedge}\simeq BU^{\wedge}$ is compatible with the $\Gal(\CC/\QQ)$ action. 

The rest of the proof follows from the proof to the Adams conjecture (see \cite[p.~69]{Sullivan1974}), which shows that the $\Gal(\CC/\QQ)$ action on $BGL(\infty,\CC)^{\wedge}\simeq BU^{\wedge}$ is the abelianized Galois action of $\widehat{\ZZ}^{\times}$ defined by the modified Adams operations.
\end{proof}

\begin{corollary}
With the same assumption as above, the following diagram also commutes up to homotopy
\[
\begin{tikzcd}
 \widehat{X} \arrow[r,"\nu_{X}"] \arrow[d,"\sigma^{-1}"] & \widehat{BU} \arrow[d,"\psi^{\omega(\sigma)}"] \\
 \widehat{X^{\sigma}} \arrow[r,"\nu_{X^{\sigma}}"] & \widehat{BU}
\end{tikzcd}
\]
where $\nu_{X}$ and $\nu_{X^{\sigma}}$ are the stable inverses of the tangent bundles $TX$ and $TX^{\sigma}$ respectively.
\end{corollary}

\begin{proof}[Proof of Lemma \ref{GaloisMainII}]

Let $X,Y$ be smooth, complete complex varieties. Assume that $Y$ is simply-connected. Let $f:X\rightarrow Y$ be an algebraic map of complex varieties representing a complex algebraic element in $\mathbf{S}(Y)^\wedge_N$, that is, $f_{\et}:X_{\et}\rightarrow Y_{\et}$ is a homotopy equivalence. Since $Y$ is simply-connected, $f_{\et}$ splits into a product of homotopy equivalences $(f_{\et})_p:(X_{\et})^{\wedge}_p\rightarrow (Y_{\et})^{\wedge}_p$.

Assume that $\omega(\sigma)= \prod_p\sigma_p\in\prod_p\widehat{\ZZ}_p^\times\simeq \widehat{\ZZ}^\times$, where $\sigma\in \Gal(\CC/\QQ)$.

Let $\nu_Y:Y\rightarrow BU$ be the stable normal bundle of $Y$. Let $\widehat{\nu_Y}_p$ be the $p$-completion of $\nu_Y$. Similarly we also have $\nu_X$ and $\widehat{\nu_X}_p$.

Recall from \cite[Theorem 6.5]{SullivanMITnotes} that $\widehat{BSTOP}_p$ is equivalent to the classifying space $B_{\widehat{KO}_p}SG^\wedge$ of profinite spherical fibrations with $\widehat{KO}_p$-orientations. The Adams conjecture states that the Galois action on $\widehat{BU}_p$ fixes the underlying $\widehat{BSG}_p$. The abelianized Galois action $\widehat{\ZZ}^{\times}$ on $\widehat{BU}_p$ extends to $\widehat{BSTOP}_p$ via the Galois action on the orientation $\widehat{KO}_p$ (\cite[Theorem 6.7]{SullivanMITnotes}). Moreover, the proof of the Adams conjecture in \cite[Theorem 6.7]{SullivanMITnotes} shows that $\widehat{BSTOP}_p\xrightarrow{\psi^{\sigma_p}-1}\widehat{BSTOP}_p$
canonically factors as a composition of maps $\widehat{BSTOP}_p\xrightarrow{g_{\sigma_p}} \widehat{G/TOP}_p\rightarrow \widehat{BSTOP}_p
$ for some $g_{\sigma_p}$. 

The following commutative diagram shows that the map $(X_{\et})^{\wedge}_p \rightarrow \widehat{BU}_p \xrightarrow{g_{\sigma_p}} \widehat{G/U}_p \rightarrow \widehat{G/TOP}_p$ corresponds to the ``surgery map'' $(X^{\sigma}_{\et})^{\wedge}_p\xrightarrow{\sigma} (X_{\et})^{\wedge}_p$. 
\begin{equation}\label{Key-Diagram}
\begin{tikzcd}
 (X_{\et})^{\wedge}_p \arrow[r,"\nu_{X}"] \arrow[d,"\sigma^{-1}"] & \widehat{BU}_p \arrow[r] \arrow[d,"\psi^{\sigma_p}"] & \widehat{BSO}_p \arrow[r] \arrow[d,"\psi^{\sigma_p}"] & \widehat{BSTOP}_p \arrow[d,"\psi^{\sigma_p}"] \\
 (X^{\sigma}_{\et})^{\wedge}_p \arrow[r,"\nu_{X^{\sigma}}"] & \widehat{BU}_p \arrow[r] & \widehat{BSO}_p \arrow[r] & \widehat{BSTOP}_p
\end{tikzcd}    
\end{equation}

\textbf{(1) Case I: $p$ is an odd prime.}

Recall $(\Delta_X)_p\in\widetilde{KO_p^\wedge}(\widehat{M(\nu_X)}_p)$ induced by $\nu_X$. Similar for $(\Delta_Y)_p$. Let $\phi_p\in \widetilde{KO_p^\wedge}(Y)$ correspond to $f:X\rightarrow Y$. Then $(f_p^{-1})^*(\Delta_X)_p=\phi_p\cdot (\Delta_Y)_p$.

By the diagram \ref{Key-Diagram}, the element $(\Delta_{X^{\sigma}})_p\in\widetilde{KO_p^\wedge}(\widehat{M(\nu_{X^{\sigma}})}_p)$ induced by stable normal bundle $\nu_{X^{\sigma}}$ of $X^{\sigma}$ is pulled back along $\sigma^{-1}$ to $\psi^{\sigma_p}(\Delta_X)_p$ over $X$. The element $\phi'_p\in \widetilde{KO_p^\wedge}(Y)$ representing $X^{\sigma}\xrightarrow{\sigma} X\xrightarrow{f} Y$ satisfies the following equation.
\[
\phi'_p\cdot (\Delta_Y)_p=(\sigma^{-1}\circ (f_p)^{-1})^*(\Delta_{X^{\sigma}})_p=(f_p^{-1})^*\psi^{\sigma_p}(\Delta_X)_p=\psi^{\sigma_p}(\phi_p\cdot(\Delta_Y)_p)
\]

It fits the $p$-part of the abelianized Galois action on $\mathbf{S}(Y)^\wedge_N$ (\ref{odd-prime-Galois-symmetry}).

\vspace{0.15in}

\textbf{(2) Case II: $p=2$.}

Let $l\in H^{4*}(X;\ZZ^{\wedge}_{2}),k\in H^{4*-2}(X;\ZZ/2)$ represent $f:X\rightarrow Y$. Let $(l',k')$ represent $f\circ \sigma:X^{\sigma}\rightarrow Y$.

Let $L_X\in H^{4*}(X;\widehat{\ZZ}_2)$ be the $2$-completion of the $2$-local Hirzebruch $L$-genus of $X$ (this is the $2$-completion of the $\ZZ_{(2)}$-coefficient $L$-class defined in \cite{Sullivan&Morgan}*{p.~530}). Then $(f_2^{-1})^*L_X=(1+8l)\cdot L_Y$ (see \cite{Sullivan&Morgan}*{Theorem 8.7}).

By the diagram \ref{Key-Diagram},
\[
(\sigma^{-1})^* L_{X^{\sigma}}=\psi^{\omega(\sigma)}_H L_X
\]
where $\psi^{\omega(\sigma)}$ is the cohomological Adams operation.

Then
\[
(1+8l')L_Y=((\widehat{f}_2\circ \sigma)^{-1})^*L_{X^{\sigma}}=(\widehat{f}_2^{-1})^*\psi^{\sigma_2}_H L_X=\psi^{\sigma_2}_H((\widehat{f}_2^{-1})^* L_X)=\psi^{\sigma_2}_H ((1+8l)\cdot L_Y)
\]
where $\psi^{\sigma_2}_H$ is the cohomological Adams operation on $H^{2*}(Y;\widehat{\ZZ}_2)$.
This is exactly the $2$-part abelianized Galois action $l$-class (\ref{2-Galois-symmetry-l-class}).

Let $k^{\sigma_2}_Y$ be the $k$-class for the map $Y^{\sigma} \xrightarrow{\sigma} Y$, which is given by $(Y_{\et})^{\wedge}_2 \rightarrow \widehat{BU}_2\xrightarrow{g_{\sigma_2}} \widehat{G/U}_2 \rightarrow \widehat{G/TOP}_2$.

Let $(l^{\sigma},k^{\sigma})$ represent $f^{\sigma}:X^{\sigma}\rightarrow Y^{\sigma}$. By the the following commutative diagram, $k^\sigma=\sigma^* k$.
\[
\begin{tikzcd}
 (X^{\sigma})_{\et} \arrow[d,"\sigma"] \arrow[r,"f^{\sigma}"] & (Y^{\sigma})_{\et} \arrow[d,"\sigma"] \arrow[r,"\nu_{Y^{\sigma}}"] & \widehat{BU}  \arrow[d,"\psi^{\omega(\sigma)^{-1}}"] \\
 X_{\et} \arrow[r,"f"] & Y_{\et} \arrow[r,"\nu_{Y}"] & \widehat{BU}
\end{tikzcd}
\]

Hence, $k'$ is also the $k$-class for the map $\sigma\circ f^{\sigma}:X^{\sigma}\rightarrow Y$. Then
\[
k'= k^{\sigma_2}_Y+ (\sigma^{-1})^* k^{\sigma}= k^{\sigma_2}_Y+k
\]
which agrees with the $2$-part abelianized Galois symmetry on the $k$-class (see \ref{2-Galois-symmetry-l-class}).
\end{proof}

\section{The characteristic class $k^{\sigma_2}$}\label{algorithm-for-k-class}

The Lemma \ref{Additivity-of-k-classes} can be deduced from \cite{Friedlander-etale-homotopy-type}*{Theorem 9.2}, where Friedlander proves that the Adams operation $\psi^{k}:BU^{\wedge}_{p}\rightarrow BG_{p}$ is an $H$-space morphism when $k$ is not divisible by $p$. We provide an alternative proof below.

\begin{proof}[Proof of \ref{Additivity-of-k-classes}]
Let $\omega':\Gal(\overline{\QQ}/\QQ)\rightarrow \widehat{\ZZ}^{\times}$ be the abelianization quotient map. Let $\alpha\in \Gal(\overline{\QQ}/\QQ)$ be the element so that $\omega'(\alpha)^{-1}=\sigma=\prod_p\sigma_p\in \widehat{\ZZ}^\times\cong \prod_p\widehat{\ZZ}^\times_p$ and $\sigma_p=1$ except for $p\neq 2$.

It suffices to consider the case when $\sigma_2$ is represented by an odd integer. It suffices that the induced map $f_\sigma:\widehat{BU}\rightarrow \widehat{G/U}$ induced by the Adams conjecture is an $H$-space map.

Recall that $\psi^{\sigma}: \widehat{BU}_2\rightarrow  \widehat{BU}_2$ is the stablization of the \'etale homotopy equivalence induced the algebraic isomorphism $\alpha: Gr_n(\CC^{N})\rightarrow Gr_n(\CC^{N})$ (see the proof of Adams conjecture in \cite[Chapter 5]{SullivanMITnotes}).
After passing $N$ to $\infty$, we have $\alpha: \widehat{BU(n)}\rightarrow \widehat{BU(n)}$.

Notice that the unstable Whitney sum $\Delta:\widehat{BU(n)}\times \widehat{BU(m)}\rightarrow \widehat{BU(n+m)}$ respects the Galois action, since it is induced from the algebraic map $Gr_n(\CC^{N})\times Gr_m(\CC^{M})r\rightarrow Gr_{n+m}(\CC^{N+M})$.

The proof of the Adams conjecture in \cite[p.~158]{SullivanMITnotes} is deduced from two facts (indeed, one needs to unravel the mathematical diagrams in terms of the inertia lemma \cite[p.~99]{SullivanMITnotes}). The first fact is that $BU(n-1)\rightarrow BU(n)$ is the universal spherical fibration of a rank $n$ vector bundle. The second is that  the following diagram commutes.
\[
\begin{tikzcd}
\widehat{BU(n-1)} \arrow[r,"\alpha"] \arrow[d] & \widehat{BU(n-1)} \arrow[d] \\
\widehat{BU(n)} \arrow[r,"\alpha"] & \widehat{BU(n)}
\end{tikzcd}
\]
Now consider the following commutative diagram
\[
\begin{tikzcd}
\widehat{BU(n+m-1)} \arrow[r,"\alpha"] \arrow[d] & \widehat{BU(n+m-1)} \arrow[d] \\
\widehat{BU(n+m)} \arrow[r,"\alpha"] & \widehat{BU(n+m)}
\end{tikzcd}
\]
It suffices that the pullback of this diagram along the H-space map $\Delta: \widehat{BU(n)}\times \widehat{BU(m)}\rightarrow \widehat{BU(n+m)}$ is equivalent, up to homotopy, to the following diagram, where $p_1,p_2$ are the projection of $\widehat{BU(n)}\times \widehat{BU(m)}$ onto the two factors and $*$ is the fiberwise join product.
\[
\begin{tikzcd}
p_1^*\widehat{BU(n-1)} * p_2^*\widehat{BU(m-1)} \arrow[r,"\alpha * \alpha"] \arrow[d] & p_1^*\widehat{BU(n-1)} * p_2^*\widehat{BU(m-1)}  \arrow[d] \\
\widehat{BU(n)}\times \widehat{BU(m)} \arrow[r,"\alpha"] & \widehat{BU(n)}\times \widehat{BU(m)}
\end{tikzcd} 
\]
It is left to check the commutativity of the following diagram.

\begin{equation}\label{*diagram}
\begin{tikzcd} 
p_1^*\widehat{BU(n-1)} * p_2^*\widehat{BU(m-1)} \arrow[r,"\alpha * \alpha"] \arrow[d] & p_1^*\widehat{BU(n-1)} * p_2^*\widehat{BU(m-1)}  \arrow[d] \\
\widehat{BU(n+m-1)} \arrow[r,"\alpha"] & \widehat{BU(n+m-1)}
\end{tikzcd}     
\end{equation}

Indeed, the map $p_1^*\widehat{BU(n-1)} * p_2^*\widehat{BU(m-1)}\rightarrow \widehat{BU(n+m-1)}$ is realized by a map $p_1^*BU(n-1) * p_2^*BU(m-1)\rightarrow BU(n+m-1)$ as follows. Each element of $BU(n-1)$ can be uniquely written as a pair of subspaces $V^{n-1}_1\subset V^n_2$ in $\CC^{\infty}$ and the map $BU(n-1)\rightarrow BU(n)$ takes $V^{n-1}_1\subset V^n_2$ to $V_2$. Now take an element $W^{m-1}_1\subset W^m_2$ in $BU(m-1)$. Let $V^{\bot}$ be the perpendicular $1$-dimensional complementary of $V^{n-1}_1\subset V^n_2$ and the same for $W^{\bot}$. There is a unit circle $\{(e^{i\phi},0)\in V^{\bot}\oplus W^{\bot}\}$ in $V^{\bot}$. Similarly $\{(0,e^{i\phi})\}$ in $W^{\bot}$. There is a family of $1$-dimensional subspaces $\{\CC_t\}_{t\in I}$ in $V^{\bot}\oplus W^{\bot}$ whose unit circles are $\{(te^{i\phi},\sqrt{1-t^2}e^{i\phi})\}$. Notice that $p_1^*BU(n-1) * p_2^*BU(m-1)$ (over $BU(n)\times BU(m)$) is a quotient of $BU(n-1) \times BU(m-1)\times I$. So the map $p_1^*BU(n-1) * p_2^*BU(m-1)\rightarrow BU(n+m-1)$ is induced by $BU(n-1) \times BU(m-1)\times I\rightarrow BU(n+m-1)$, which maps $(V^{n-1}_1\subset V^n_2,W^{m-1}_1\subset W^m_2)$ to $(V_1\oplus W_1\oplus \CC_t)\subset (V_2\oplus W_2)$.

Moreover, the map $BU(n-1) \times BU(m-1)\times I\rightarrow BU(n+m-1)$ is homotopic to the stablization of a map $f:Gr_{n-1}(\CC^{N})\times Gr_{m-1}(\CC^{M})\times I\rightarrow Gr_{m+n-1}(\CC^{N+M+2})$ with a similar construction like above. Notice that $\CC^{N+M+2}=\CC^2\oplus \CC^N\oplus \CC^M$. There are two unit circles in the axes of $\CC^2$, namely $\{(e^{i\phi},0)\}$ and $\{(0,e^{i\phi})\}$. Then there is a family of $1$-dimensional subspaces $\{\CC'_{t}\}_{t\in I}$ of $\CC^2$, whose unit circles are $\{(te^{i\phi},\sqrt{1-t^2}e^{i\phi})\}$. Given a subspace $V^{n-1}$ in $\CC^N$ and a subspace $W^{m-1}$ in $\CC^M$, $f_t(V,W)=\CC'_t\oplus V\oplus W$ in $\CC^{N+M+2}$. 

Indeed, the map $f:Gr_{n-1}(\CC^{N})\times Gr_{m-1}(\CC^{M})\times I\rightarrow Gr_{m+n-1}(\CC^{N+M+2})$ can be extended to a map $Gr_{n-1}(\CC^{N})\times Gr_{m-1}(\CC^{M})\times Gr_{1}(\CC^2)\rightarrow Gr_{m+n-1}(\CC^{N+M+2})$ induced by the direct sum of subspaces. Under the $Gr_{1}(\CC^2)\cong \CC P^1$, we embed $I$ as the half real line $[0,\infty]$ in $\CC P^1$.

Since  $Gr_{n-1}(\CC^{N})\times Gr_{m-1}(\CC^{M})\times \CC P^1 \rightarrow Gr_{m+n-1}(\CC^{N+M+2})$ is an algebraic map defined over $\ZZ$, we have the following commutative diagram
\[
\begin{tikzcd}
\widehat{BU(n-1)}\times \widehat{BU(m-1)}\times \widehat{\CC P^{1}} \arrow[r,"\alpha \times \alpha\times \alpha"] \arrow[d] & \widehat{BU(n-1)}\times \widehat{BU(m-1)}\times \widehat{\CC P^{1}}  \arrow[d] \\
\widehat{BU(n+m-1)} \arrow[r,"\alpha"] & \widehat{BU(n+m-1)}
\end{tikzcd}
\]
However, notice that the map $\alpha$ on $\widehat{\CC P^{1}}$ is in fact homotopic to the completion of the map $z\rightarrow z^{\sigma}$ on $\CC P^{1}$, since the homotopy classes of self homotopy equivalences of $\widehat{\CC P^{1}}$ is determined by the induced group homomorphism on $H^{2}(\widehat{\CC P^{1}};\widehat{\ZZ})$. 

But under the embedding $[0,\infty]\subset \CC P^1$, the restriction of the map $z\rightarrow z^{\sigma}$ to $[0,\infty]$ is homotopic to the identity map. Hence, we get the following commutative diagram 
\[
\begin{tikzcd}
\widehat{BU(n-1)}\times \widehat{BU(m-1)}\times I \arrow[r,"\alpha \times \alpha\times 1"] \arrow[d] & \widehat{BU(n-1)}\times \widehat{BU(m-1)}\times I  \arrow[d] \\
\widehat{BU(n+m-1)} \arrow[r,"\alpha"] & \widehat{BU(n+m-1)} 
\end{tikzcd}
\]
Passing to the quotient of the spaces in the upper horizontal arrow, it is exactly the diagram \ref{*diagram}.
\end{proof}

\begin{example}
We compute the class $k^{\sigma_2}$ of the stable normal bundle $\nu_{\CC P^N}$ of $\CC P^N$ with $N$ even. The additivity of $k^{\sigma_2}$ implies that $k^{\sigma_2}_{\CC P^N}=k^{\sigma_2}(T\CC P^N)=k^{\sigma_2}(\nu_{\CC P^N})$.

Let $\omega':\Gal(\overline{\QQ}/\QQ)\rightarrow \widehat{\ZZ}^\times$ be the abelianization quotient map. Let $\alpha\in \Gal(\overline{\QQ}/\QQ)$ be an element such that
$\omega'(\alpha)^{-1}=\sigma\prod_p\sigma_p\in\prod_p\widehat{\ZZ}^{\times}_p\cong\widehat{\ZZ}^\times$. It suffices to consider the case when $\sigma_2$ is represented by an integer and all other $\sigma_p=1$.

Any homotopy class of self \'etale homotopy equivalence of $\CC P^N$ is determined by a map on $H^{2}(\widehat{\CC P^N};\widehat{\ZZ})$. Hence, the Galois automorphism $\alpha$ on $\widehat{\CC P^N}$ is homotopic to the map $f_{\sigma_2}([x_0,\cdots,x_N])=[x^{\sigma_2}_0,\cdots, x^{\sigma_2}_N]$ (also see \cite[Corollary 5.4]{SullivanMITnotes}).

As in \cite[Theorem 9]{Sullivan1996}, the element in $\mathbf{S}^{TOP}(\CC P^N)$ is determined by the `splitting invariants' on the submanifolds $\CC P^n$ for $n=1,2,\cdots,N-1$. As a result, the associated Kervaire class $k^{\sigma_2}_{\CC P^N}$ is determined by the $2$-adic Kervaire invariant of $f_{\sigma_2}$ on $\CC P^n$ for $n$ odd and $n\geq 3$, namely, the Kervaire invariant is $\langle k^{\sigma_2}_{\CC P^N},\CC P^n\rangle$ (see \cite[p.~91, Proof of Theorem 4']{Sullivan1996}).

Since the transversal preimage of $\CC P^n$ can be made into a complete intersection of several degree $\sigma_2$ hypersurfaces, by Lefschetz's theorem we know that $H_i(f_{\sigma_2}^{-1}(\CC P^n);\ZZ/2)\rightarrow H_i(\CC P^n;\ZZ/2)$ is an isomorphism for $i\neq n$. \cite{Wood1975}\cite{Wood1979}\cite{Browder1978} show that the Kervaire invariant of a complete intersection $V^k$ in a complex projective space obstructs to finding a symplectic basis $\alpha_i$ for $H_k(V;\ZZ/2)$ so that $V$ is the connected sum of a manifold with the same homology like $\CC P^k$ and several $S^k\times S^k$ indexed by $\alpha_i$. So their Kervaire invariant of $f_{\sigma_2}^{-1}(\CC P^n)$ is exactly the Kervaire invariant for the map $f_{\sigma_2}$, namely, the obstruction to finding some surgery process on $f_{\sigma_2}^{-1}(\CC P^n)$ such that its $\ZZ/2$-homology is isomorphic to that of $\CC P^n$.

When $n\neq 1,3,7$, the Kervaire invariant of $f_{\sigma_2}^{-1}(\CC P^n)$ is the modified Legendre symbol valued in $\ZZ/2$, i.e., 
\[
\langle k^{\sigma_2}_{\CC P^N},[\CC P^n]\rangle=\legendre{2}{\sigma_2} =
    \begin{cases}
        0 & \text{if } \sigma_2  \equiv \pm 1 \pmod{8} \\
        1 & \text{if } \sigma_2  \equiv \pm 3 \pmod{8} 
    \end{cases}
\]
Hence, if $\omega\in H^2(\CC P^N;\ZZ/2)$ is the generator, then the $n$-th component of $k^{\sigma_2}_{\CC P^N}$ is $\legendre{2}{\sigma_2}\omega^n$.

When $n=1,3,7$, the Kervaire invariant vanishes.

For $n=1$, we need to use an alternative definition for the Kervaire invariant. One can homotope the map $f_{\sigma_2}$ such that $f^{-1}_{\sigma_2}(\CC P^1-\pt)=f^{-1}_{\sigma_2}(\CC P^1)-\pt$. Let $\nu$ be the normal bundle of $\CC P^1$ in $\CC P^N$. Choose a framing on $\nu\vert_{\CC P^1 -\pt}$, namely, a map $\CC P^1 -\pt\rightarrow SO(2N-2)$. It induces a framing on $f^{-1}_{\sigma_2}(\CC P^1)-\pt$, namely, $f^{-1}_{\sigma_2}(\CC P^1)-\pt\rightarrow \CC P^1 -\pt\rightarrow SO(2N-2)$. We need to check whether the framed manifold $f^{-1}_{\sigma_2}(\CC P^1)-\pt$ is zero or not in the almost framed bordism group $P_2$. Notice that $\nu$ has a complex structure, so we can choose a framing which factors through $SU(N-1)$, i.e., $\CC P^1 -\pt\rightarrow SU(N-1)\rightarrow SO(2N-2)$. So the framing on $f^{-1}_{\sigma_2}(\CC P^1)$ factors through a `$SU(N-1)$-framing'. However, $\pi_1(SU(N-1))=0$, so the framing on $f^{-1}_{\sigma_2}(\CC P^1)$ has no twisting. That is, the Kervaire invariant on $f^{-1}_{\sigma_2}(\CC P^1)$ is $0$. So $\langle k^{\sigma_2}_{\CC P^N},[\CC P^1]\rangle=0$.
\qed
\end{example}

Let $\gamma$ be the universal complex line bundle on $\CC P^{2N}$. Notice that the normal bundle $\nu_{\CC P^{2N}}$ is isomorphic to $(2N+1)\gamma^*$, where $\gamma^*$ is the complex dual bundle of $\gamma$. The additivity of $k^{\sigma_2}$ implies that $k^{\sigma_2}_{\CC P^{2N}}=k^{\sigma_2}(\gamma)$. In particular, $k^{\sigma_2}(\gamma)$ is irrelavant to $N$. So we may let $N$ by the infinity.

Let $x_1,x_2,\cdots$ (of degree $2$) be the roots of the Stiefel-Whitney classes induced from $BU(1)\times BU(1)\times \cdots\rightarrow BU$. Again, by the additivity of $k^{\sigma_2}$ class, we can write 
\[
k^{\sigma_2}=k^{\sigma_2}_1(x_1+x_2+\cdots)+k^{\sigma_2}_3(x_1^3+x_2^3+\dots)+\dots
\]
where each $k^{\sigma_2}_i\in \ZZ/2$ can be calculated by the previous example. So we have proved the following:

\begin{proposition}
\[
 k^{\sigma_2}_{2i+1} =
    \begin{cases}
        \legendre{2}{\sigma_2} & \text{if } 2i+1\neq 1,3,7 \\
        0 & \text{if } 2i+1=1,3,7
    \end{cases}
\]   
\end{proposition}

\bibliographystyle{amsalpha}
\bibliography{ref}

\Addresses

\end{document}